\documentclass[10pt]{amsart}
\usepackage{amssymb}
\usepackage{color}
\usepackage{graphicx}
\usepackage{float}
\usepackage[all,cmtip]{xy}
\usepackage{tikz}
\usetikzlibrary{matrix}

\newcommand{\SE}{{\mathcal{E}}}
\newcommand{\SD}{{\mathcal{D}}}
\newcommand{\SW}{{\mathcal{W}}}

\newcommand{\SR}{{\mathcal{R}}}
\newcommand{\SF}{{\mathcal{F}}}

\newcommand{\HI}{{\mathcal{H}\mathcal{I}}}
\newcommand{\FHI}{{\mathcal{F}\mathcal{H}\mathcal{I}}}
\newcommand{\Imm}{{\operatorname{\mathcal{I}}}}
\newcommand{\FImm}{{\operatorname{\mathcal{F}\mathcal{I}}}}
\newcommand{\Maps}{\operatorname{\mathcal{M}}\mathfrak{aps}}
\newcommand{\Leg}{{\mathcal{L}eg}}
\newcommand{\FLeg}{{\mathcal{F}\mathcal{L}eg}}

\mathchardef\mhyphen="2D

\newcommand{\Engel}{{\operatorname{\mathfrak{Engel}}}}
\newcommand{\Geiges}{{\operatorname{Geiges}}}

\newcommand{\std}{{\operatorname{std}}}
\newcommand{\lorentz}{{\operatorname{lorentz}}}
\newcommand{\niso}{{\operatorname{n.e.t.}}}
\newcommand{\gen}{{\operatorname{gen}}}

\newcommand{\Mon}{\operatorname{Mon}}

\newcommand{\R}{{\mathbb{R}}}
\newcommand{\Z}{{\mathbb{Z}}}

\newcommand{\NS}{{\mathbb{S}}}
\newcommand{\D}{{\mathbb{D}}}
\newcommand{\PP}{{\mathbb{P}}}
\newcommand{\Op}{{\mathcal{O}p}}
\newcommand{\ST}{{\mathcal{T}}}
\newcommand{\SU}{{\mathcal{U}}}

\newtheorem{lemma}{Lemma}
\newtheorem{proposition}{Proposition}
\newtheorem{theorem}{Theorem}
\newtheorem*{theorem*}{Theorem}
\newtheorem{corollary}{Corollary}
\newtheorem{definition}{Definition}

\theoremstyle{definition}

\newtheorem{remark}{Remark}

\newtheorem{example}{Example}

\begin{document} 

\title{Flexibility for tangent and transverse immersions in Engel manifolds}

\date{\today}


\author{\'Alvaro del Pino}
\address{Universidad Aut\'onoma de Madrid and Instituto de Ciencias Matem\'aticas -- CSIC. C. Nicol\'as Cabrera, 13--15, 28049, Madrid, Spain.}
\email{alvaro.delpino@icmat.es}

\author{Francisco Presas}
\address{Instituto de Ciencias Matem\'aticas CSIC-UAM-UC3M-UCM, C. Nicol\'as Cabrera, 13-15, 28049 Madrid, Spain}
\email{fpresas@icmat.es}

\begin{abstract}
In this article we study immersions of the circle that are tangent to an Engel structure $\SD$. We show that a full $h$--principle does exist as soon as one excludes the closed orbits of $\SW$, the kernel of $\SD$. This is sharp: we elaborate on work of Bryant and Hsu to show that curves tangent to $\SW$ often conform additional isolated components that cannot be detected at a formal level. We then show that this is an exceptional phenomenon: if $\SD$ is generic, curves tangent to $\SW$ are not isolated anymore.

We then go on to show that a full $h$--principle holds for immersions transverse to the Engel structure.
\end{abstract}

\maketitle

\section{History of the problem and outline of the paper}

In \cite[p. 84]{Gr}, M. Gromov stated the following result and posed it as an exercise for the reader: the sheaf of \emph{tangent} $C^\infty$--immersions $\mathbb{R} \to (M, \xi)$, with $\xi$ a bracket generating distribution, is microflexible. He then sketched an argument to show that this implies that the sheaf of $C^\infty$--immersions of some given manifold $N^n$ into $M^m$, $n <m$, that are \emph{transverse} to $\xi$ is flexible. 

Loosely speaking, the geometrical intuition behind the microflexibility of tangent immersions is as follows: since $\xi$ is bracket generating, any deformation (not necessarily tangent) of a tangent immersion cam be approximated by integrating directions contained in $\xi$. However, perhaps surprisingly, R. Bryant and L. Hsu proved in \cite{BH} that this cannot possibly hold in the $C^1$--topology: They were able to show that rigid tangent segments -- segments that, relative to their ends, are isolated in the $C^1$ topology -- exist in abundance in maximally non--integrable $2$--distributions in dimension $4$ and onwards.

The aim of this article is twofold. First of all, as we prove in Subsection \ref{ssec:BH} and then further elaborate in Section \ref{sec:shortCurves}, the existence of rigid segments poses a problem for a complete $h$--principle to hold for immersions tangent to an Engel structure $\SD$. We show that, by restricting to those immersions that are somewhere not tangent to $\SW$, this result can be salvaged: this is the content of Theorem \ref{thm:main}. A key ingredient in its proof is Theorem \ref{thm:generic}, where we show that generic families of tangent curves not everywhere tangent to $\SW$ are in general position with respect to $\SW$. In Section \ref{sec:shortCurves}, we discuss deformations of curves tangent to $\SW$ and we show that, if $\SD$ is generic, Theorem \ref{thm:main} can be strengthened to a full $h$--principle at the $\pi_0$ level, see Theorem \ref{thm:main2}.

In Section \ref{sec:transverse}, we show that Gromov's result for transverse maps and immersions holds in the Engel case. Mostly, this follows from a simple remark and a straightforward application of Gromov's methods, as explained in \cite{Gr, EM}, but the case of $3$--dimensional immersions does require for us to apply Theorem \ref{thm:main}.

In both the transverse and the tangent cases, our methods probably generalise to prove analogous statements for higher dimensional Goursat structures.

\textbf{Acknowledgements:} The authors are grateful to T. Vogel for bringing the problem of transverse submanifolds to their attention and for enlightening discussions, to R. Casals for his interest in the development of this paper, and to V. Ginzburg for the many conversations that gave birth to this project. The authors are supported by the Spanish Research Projects SEV--2015--0554, MTM2013--42135, and MTM2015--72876--EXP and a La Caixa--Severo Ochoa grant.

\section{Preliminaries} \label{sec:prelim}

\subsection{Elementary facts about Engel structures}

\begin{definition}
Let $M$ be a smooth closed $4$--manifold and let $\SD \subset TM$ be a smooth $2$--distribution. $\SD$ is said to be \emph{Engel} if it is maximally non--integrable, that is, $\SE = [\SD, \SD]$ is a $3$--distribution and $[\SE, \SE] = TM$. The pair $(M, \SD)$ is said to be an \emph{Engel manifold}. 

$\SE$ is called the associated \emph{even--contact structure}. The line field $\SW \subset \SD$ uniquely defined by $[\SW, \SE] \subset \SE$ is called the \emph{kernel} of the Engel structure. The flag $\SW \subset \SD \subset \SE \subset TM$ is called the \emph{Engel flag} associated to $\SD$. 
\end{definition}

It is not hard to check that $\SE$ is orientable and always comes with a canonical orientation. However, this is not the case for $\SW$, $\SD$ and $TM$. We will then speak of oriented and orientable Engel flags. 

A relevant feature of Engel structures is that they possess a local model:

\begin{proposition}[F. Engel]
Let $(M_i, \SD_i)$, $i=1,2$, be two Engel manifolds. Let $p_i \in M_i$, $i=1,2$. Then there are neighbourhoods $U_i \ni p_i$ and a diffeomorphism $\phi: U_1 \to U_2$ such that $\phi^*\SD_2 = \SD_1$.
\end{proposition}

The following local models will be useful in subsequent discussions:

\begin{example} \label{example:darboux}
Let $M = \R^4$ with coordinates $(x,y,z,t)$. The $2$--distribution $\SD_{\std} = \ker(dy-zdx) \cap \ker(dz-tdx)$ is sometimes called the \emph{standard Engel structure} or the \emph{Darboux model}. It corresponds to the tautological distribution in the jet space $J^2(\R, \R) \cong \R^4$, so $x$ can be thought as a variable, $y$ as a function on $x$, and $z$ and $t$ as the first and second derivatives, respectively.

Here $\SW$ is spanned by $\partial_t$. $\SE = [\SD_{\std},\SD_{\std}]$ induces on every slice $\R^3 \times \{t_0\}$ the contact structure $\ker(dy-zdx)$. \hfill{$\Box$}
\end{example}

\begin{example}
Take $M = \R^4$ with coordinates $(x,y,z,t)$ again. The $2$--distribution 
\[ \SD_{\lorentz} = \ker(dy-tdx) \cap \ker(dz-t^2dx) = \langle \partial_t, \partial_x + t\partial_y + t^2\partial_z \rangle\]
is an Engel structure, which we call the \emph{Lorentzian model}. $\SW$ is spanned by $\partial_x + t\partial_y + t^2\partial_z$.\hfill{$\Box$}
\end{example}

The following construction first appeared in works of Cartan:

\begin{example}
Fix a contact manifold $(M^3, \xi)$. Denote the projectivisation of the bundle $\xi$ as $\pi: \PP(\xi) \to M$. Define the \emph{Cartan contact prolongation} (of one projective turn) as the manifold $\PP(\xi)$ equipped with the $2$--distribution
\[ \SD(\xi)(p,L) = (d\pi)^{-1}(L) \]
where $p$ is a point in $M$, $L$ a line in $\xi_p$, and $(p,L)$ a point in $\PP(\xi)$. $\SD(\xi)$ is an Engel structure with kernel $\SW(\xi) = \ker(d\pi)$. The even--contact structure associated to it is $(d\pi)^{-1}(\xi)$. Observe that the Darboux model can be understood as a Cartan prolongation where $\R^3 \times \{\infty\}$ has been removed. \hfill{$\Box$}
\end{example}

We denote by $\Engel(M)$ the space of Engel structures in $M$ endowed with the $C^2$--topology. By the existence $h$--principle in \cite{CPPP}, Engel structures do exist in $M$ as soon as there are no obvious topological obstructions.

\subsection{Horizontal immersions}

Recall that all submanifolds tangent to an Engel structure have dimension at most 1. This motivates the interest in the following definition:

\begin{definition}
Let $(M, \SD)$ be an Engel manifold. An immersion $\gamma: \NS^1 \to M$ is said to be \emph{horizontal} if $\gamma'(t) \in \SD_{\gamma(t)}$ for all $t \in \NS^1$.

When $\gamma$ is an embedding, we say that it is an \emph{Engel knot} or a \emph{horizontal knot}.
\end{definition}

Following the $h$--principle philosophy, one can understand an immersion as two separate maps, the map itself and its derivative, that are coupled together. Decoupling this relation yields the following definition:

\begin{definition}
Let $(M, \SD)$ be an Engel manifold. A \emph{formal horizontal immersion} is a pair $(\gamma, F)$ satisfying:
\begin{enumerate}
\item $\gamma: \NS^1 \to M$ is a smooth map,
\item $F: \NS^1 \to \SD|_\gamma$ is a non--vanishing section $F(s) \in \SD_{\gamma(s)} \subset T_{\gamma(s)} M$.
\end{enumerate}
\end{definition}

Write $\Imm$ for the space of immersions of $\NS^1$ into $M$, equipped with the $C^1$--topology. Write $\FImm$ for the space of formal immersions (pairs $(\gamma,F)$ with $F$ mapping into $TM|_\gamma$), endowed with the $C^0$--topology. Denote by $\HI(\SD) \subset \Imm$ and $\FHI(\SD) \subset \FImm$ the subspaces of horizontal and formal horizontal immersions, respectively. There is a natural continuous inclusion $\HI(\SD) \to \FHI(\SD)$. Since these definitions make sense as well for immersions of the interval, we define $\Imm([0,1])$, $\FImm([0,1])$, $\HI([0,1],\SD)$, and $\FHI([0,1],\SD)$ analogously.

\begin{remark} \label{rem:formal}
Observe that there is a forgetful map $\FHI(\SD) \to \Maps(\NS^1, M)$. The fibre over $\gamma \in \Maps(\NS^1, M)$ is $\Mon(T\NS^1,\SD|_\gamma)$, where $\Mon$ denotes bundle monomorphisms. In particular, depending on whether $\SD$ is orientable or not over $\gamma$, this corresponds to $\Mon_{\NS^1}(\R,\R^2)$ or $\Mon_{\NS^1}(\R,\R\oplus L)$, with $L$ the non orientable line bundle over the circle. In any case, the map is indeed a locally trivial fibration.

If $\SD$ is orientable and a global framing has been chosen for it, it follows that:
\[ \pi_0(\FHI(\SD)) \cong \pi_0(\Maps(\NS^1, M)) \times \Z \]
where the integer represents the turning of $F$ with respect to the framing. This integer is usually called the {\em rotation} number of the formal immersion. Similarly, take $(\gamma, F) \in \FHI(\SD)$:
\[ \pi_1(\FHI(\SD), (\gamma, F)) \cong \pi_1(\Maps(\NS^1, M, [\gamma])) \times \Z \]
where, if we have a loop $(\gamma_\delta, F_\delta)$, $\delta \in \NS^1$, in $\FHI(\SD)$, the integer represents the turning of $\delta \to F_\delta(0)$ with respect to the framing of $\SD$. Finally:
\[ \pi_j(\FHI(\SD), (\gamma, F)) \cong \pi_j(\Maps(\NS^1, M, [\gamma])) \text{ for } j>1. \]  \hfill{$\Box$}
\end{remark}

The aim of this article is understanding the nature of the inclusion $\HI(\SD) \to \FHI(\SD)$. In Subsection \ref{ssec:BH} we will strengthen a result of Bryant and Hsu to prove that this cannot be a weak homotopy equivalence. Then, we will restrict to a subset of horizontal immersions for which the inclusion into $\FHI(\SD)$ is indeed a weak homotopy equivalence.

\subsection{The developing map}

Let $(M, \SD)$ be an Engel manifold and fix a point $p \in M$. Denote by $\phi$ the germ of an embedding of a disk transverse to $\SW$ and passing through $p$ at the origin. After suitable reparametrisation, $\phi^*\SE$ and $\phi^*\SD$ can be assumed to be the germs $\xi_0 = \ker(dy-zdx)$ and $\langle \partial_z \rangle$, respectively.

If we assume for the Engel flag to be oriented, we can find a vector field $W$ positively spanning $\SW$ everywhere. Denote by $\varphi_q(t)$ its time $t$ flow starting from $q$. Then, any small disk representing $\phi$ can be extended to an immersion $\Psi_p: \D^3 \times \R \to M$ using the equation $\Psi_p(x,y,z,t)=\varphi_{\phi(x,y,z)}(t)$. 

Recall that the flow of $\SW$ preserves $\SE$ and hence the contact structure induced by $\SE$ on every $\Psi_p(\D^3 \times \{t\})$ is precisely $\xi_0$. This implies that $\Psi_p^* \SD|_{\{0\} \times \R}$ is then spanned by the vector fields $\partial_t$ and $X = \cos(f(x,y,z,y))\partial_z + \sin(f(x,y,z,t)) (z\partial_y+\partial_x)$, with $f: \D^3\times \R \to \R$ smooth. By the Engel condition, it holds that $\partial_t f > 0$; by construction we have $f(x,y,z,0) = 0$. 

\begin{definition}
The map 
\begin{eqnarray*}
\hat{p}: \R & \to &\PP(\xi_0) \\
t & \to & \PP(X(\Psi_p(0,t)))
\end{eqnarray*}
is called the developing map at $p$.
\end{definition}

Observe that the developing map is independent of the representative of $\phi$ we choose, but does depend on $W$ and $\phi$. However, its integer part, that measures the number of {\em full turns} (which is often unbounded) is well defined and independent of $W$ and of $\phi$. In particular, since it is well defined at the level of germs, the number of integer turns can be defined even if $\SW$ is not orientable.

If a curve $\gamma:[0,1] \to M$ tangent to $\SW$ is given, one can assign to it its developing map 
\[ \hat{\gamma}:[0,1] \to \R\PP^1 \]
given by the restriction of the developing map of $\gamma(0)$:
\[ \hat{\gamma}(s) = \widehat{\gamma(0)}(\varphi_{\gamma(0)}^{-1}(\gamma(s))) \]
and its number of integer turns does not depend on the explicit parametrisation of $\gamma$ nor on the choice of $\phi$. This can be done even if the Engel flag is not orientable.

Note that we are always referring to the developing map as mapping into the \textsl{projectivisation} of the contact structure. This implies that when we say that it is making a full turn, we mean only half a turn into the oriented \textsl{sphere bundle} of the contact structure.

\subsection{The Geiges projection}

Let us focus on horizontal immersions and embeddings into the standard Engel structure $(\R^4, \SD_{\std})$. The following result was already proven by Adachi \cite{Ada} and Geiges \cite{Ge}:

\begin{proposition} \label{prop:Geiges}
Horizontal knots in $(\R^4, \SD_{\std})$ are classified, up to homotopy, by their rotation number with respect to $\SW$. 
\end{proposition}

First of all, observe that in dimension $4$ there is no smooth knotting of $\NS^1$ and, mimicking the discussion in Remark \ref{rem:formal}, the only formal invariant is precisely the rotation number. 

Let us now outline Geiges' approach; note that our naming convention for the variables is different from his. Take coordinates $(x,y,z,t)$ in $\R^4$, $\SD_{\std}$ is given as the kernel of the $1$--forms $\alpha = dy-zdx$ and $\beta = dz-tdx$. We say that the map $\pi_\Geiges: (x,y,z,t) \to (x,z,t)$ is the \textsl{Geiges projection}. Horizontal immersions are clearly tranverse to the projection direction and, therefore, the Geiges projection takes horizontal immersions to Legendrian immersions for $\ker(\beta)$. Its front projection $(x,z,t) \to (x,z)$ is the lagrangian projection for $(\R^3(x,y,z), \ker(\alpha))$. In particular, the image of an horizontal immersion $\gamma: \NS^1 \to \R^4$ under the Geiges projection is a legendrian immersion that satisfies the additional area constraint $\int_\gamma zdx = 0$. The rotation number of $\gamma$ agrees precisely with the rotation of its Geiges projection. 

The key observation is that, given two horizontal immersions, we can use the $h$--principle for legendrian immersions first to find a homotopy between their Geiges projections, and then adjust areas along the homotopy to ensure the area constraint in the family. This gives the result for immersions. A genericity argument for areas shows that one can avoid self--intersections of the curves, giving the result for embeddings.

\begin{remark}
In Geiges' argument, one uses the fact that the knots live in $\R^4$ to perform the adjustment of areas (guaranteeing that the knots in the Geiges projection do lift to actual horizontal curves). Since Engel Darboux balls are ``smaller'' than a full $\R^4$, Theorem \ref{thm:main} does not follow immediately from his arguments.

On the other hand, it is easy to see that the genericity argument used to prove Proposition \ref{prop:Geiges} does not work in the classification of $1$--parametric families of embeddings anymore: self--intersections generically appear conforming a codimension--$2$ stratum within a family.
\end{remark}

\subsection{Local models and deformations of tangent curves} \label{ssec:models}

Let $(M, \SD)$ be an Engel manifold and let $\gamma: [-1,1] \to M$ be an horizontal immersion with $\nu(\gamma)$ a small tubular neighbourhood. Consider the following problem: to describe the deformations of $\gamma$, relative to its ends, as an horizontal curve.

\begin{lemma} \label{lem:transverseModel}
Let $\gamma$ be horizontal and everywhere transverse to $\SW$. Then, there is a map:
\[ \phi: (\nu(\gamma), \SD) \to (\R^4, \SD_{\std}) \]
which is an Engel isomorphism with its image and that satisfies  $(x,0,0,0) = \phi \circ \gamma(x)$.

In particular, all $C^1$--perturbations of $\gamma$ are, in the model, of the form $(x,y(x),y'(x),y''(x))$, with $y(x)$ vanishing to order three at the boundary points. 
\end{lemma}
\begin{proof}
Take two linearly independent vector fields $Y,Z$ along $\gamma$ such that $\langle \gamma', Y, Z \rangle$ is a $3$--distribution complementary to $\SW$. Use the exponential map on $\langle Y,Z\rangle$ along $\gamma$ to find an embedding $\psi: [-1, 1] \times \D^2_\varepsilon \to M$; since $\psi^*\SE$ is a contact structure on this slice, with $\psi^{-1} \circ \gamma$ a legendrian curve, we can take a small tubular neighbourhood of $\gamma$ within the slice that is contactomorphic to $([-1,1] \times \D^2, \ker(dy-zdx))$ with $\gamma$ corresponding to $y=z=0$.

The legendrian line field determined on the slice by $\psi^*\SD$ is not spanned by $\partial_x + z \partial_y$ necessarily. However, since this is the case along $\gamma$, we can wiggle $\psi$ slightly and restrict to a smaller tubular neighbourhood to achieve it. Now we take $W$ spanning $\SW$ and flow the slice. Adjusting the length of $W$ appropriately yields the desired model. 

The claim about the deformations follows from the fact that this is precisely the structure in the jet space $J^2(\R, \R)$ along the zero section.
\end{proof}

At the other side of the spectrum, we have the following:
\begin{lemma} \label{lem:tangentModel}
Let $\gamma$ be an integral curve of $\SW$ with developing map less than a turn. Then there is an Engel embedding:
\[ \phi: (\nu(\gamma), \SD) \to (\R^4, \SD_{\lorentz})  \]
Where $\phi\circ\gamma(s) = (x(s),0,0,0)$, with $x: [-1, 1] \to \R$ a diffeomorphism with its image satisfying $x(0) = 0$.

Any $C^1$--perturbation of $\gamma$ is given as $(x(s),y(x(s)),z(x(s)),t(x(s)))$ with 
\[ y(x) = y(0) + \int_0^x t(s) ds \]
\[ z(x) = z(0) + \int_0^x t^2(s) ds \]
In particular, the expression for $z$ implies that there are no deformations relative to the ends.
\end{lemma}
\begin{proof}
By using the flow of $\partial_x + t\partial_y + t^2\partial_z$, we can obtain the following change of coordinates:
\[ \Psi: \R^4 \to \R^4 \]
\[ \Psi(x,y,z,t) = (x, y+ tx, z+t^2x, t) \]
It is easy to check that it satisfies:
\[ \Psi^*(\SD_{\lorentz} = \ker(dy-tdx) \cap \ker(dz-t^2dx)) = \SD_0 = \ker(dz + xdt) \cap \ker(dy -2tdz) \] 
which can be thought as a Cartan prolongation of one projective turn of the contact structure $\ker(dy -2tdz)$ with the slice at infinity removed.

Since $\gamma$ is an integral curve of $\SW$, there is a map:
\[ \Phi: (\nu(\gamma), \SD) \to (\R^4, \SD_0) \]
taking $\gamma$ to some interval of the line $\R \times \{0\}$. This follows by first finding some transverse disk to $\SW$ passing through $\gamma(0)$ and parametrising it so that the contact structure induced by $\SE$ is precisely $\ker(dy -2tdz)$ and then using the flow of a vector field spanning $\SW$. The desired map is $\Psi\circ\Phi$, from the construction it is guaranteed that $x(0) = 0$. The claim regarding deformations is immediate from the model.
\end{proof}

\begin{remark}
It follows from their proofs that Lemmas \ref{lem:transverseModel} and \ref{lem:tangentModel} hold parametrically.
\end{remark}

\subsection{A result of Bryant and Hsu} \label{ssec:BH}

According to Lemma \ref{lem:tangentModel}, sufficiently short integral curves of $\SW$ are isolated in the $C^1$--topology. This phenomenon was fully characterised by Bryant and Hsu:
\begin{proposition}[Proposition 3.1 in \cite{BH}] \label{prop:BH}
Let $(M, \SD)$ be an Engel manifold. Then, an horizontal immersion $\gamma: [0,1] \to M$ is isolated if and only if it is everywhere tangent to $\SW$ and its associated developing map is injective away from its endpoints.
\end{proposition}

We can slightly extend their result to actually show that curves tangent to $\SW$ do provide a counter--example to the $h$--principle for horizontal immersions. This shows that the results presented in Section \ref{sec:main} are sharp.

\begin{proposition} \label{prop:BH2}
In a Cartan prolongation $(\PP(\xi), \SD(\xi))$ of $1$ projective turn, the embedded curves tangent to $\SW(\xi)$ are isolated, in the $C^1$--topology, inside the space of horizontal immersions. They conform two connected components that deformation retract to $\PP(\xi)$.
\end{proposition}

\begin{proof}
Everything is reduced to showing that the only $C^1$--perturbations of such a curve $\gamma$ are the nearby curves tangent to $\SW(\xi)$. Take some curve $\eta$ that is $C^1$--close to $\gamma$. If it is tangent to a $\SW(\xi)$--orbit $\tilde\gamma$ in an open interval $I \subset \NS^1$, then $\eta|_{\NS^1 \setminus I}$ is a compactly supported deformation of $\tilde\gamma|_{\NS^1 \setminus I}$, which makes less than one projective turn. Applying Lemma \ref{lem:tangentModel} implies that $\eta$ is a reparametrisation of $\tilde\gamma$. 

Otherwise, take some time $t_0$ where $\eta$ is transverse to $\SW(\xi)$. Lemma \ref{lem:transverseModel} yields a neighbourhood $U$ of $\eta(t_0)$. The vector field $\eta'(t)|_{[t_0-\delta, t_0+\delta]}$, for $\delta$ small, can be extended in $U$ to a vector field $X$ whose flowlines are $C^1$--close to $\eta$ (and thus graphical over $\gamma$). Now, in $\Op(\eta(t_0)) \subset U$, we perturb $X$ to make it tangent to $\SW(\xi)$. A suitable choice yields flowlines that are $C^1$--close to $\eta$ in $\Op(\partial U)$, that are still graphical over $\gamma$, and that are tangent to $\SW$ in an interval. Lemma \ref{lem:transverseModel} implies that we can take one such flowline and interpolate back to $\eta$ in $\Op(\partial U)$ to yield a curve $\tilde\eta$ that is a $C^1$--small deformation of $\gamma$. The proof concludes by applying to $\tilde\eta$ the reasoning in the first paragraph.
\end{proof}

In more general Engel manifolds it is not always true that sufficiently short curves tangent to $\SW$ are isolated as tangent immersions. This is explored in Section \ref{sec:shortCurves}.

\section{The $h$--principle for horizontal immersions} \label{sec:main}


Proposition \ref{prop:BH2} motivates us to focus our attention in the following class of curves:

\begin{definition}
We denote by $\HI^{\niso}(\SD) \subset \HI(\SD)$ the open subspace of those horizontal curves that are not everywhere tangent to $\SW$.
\end{definition}

Note that the space of unparametrised immersed curves that are everywhere tangent to $\SW$, the closed orbits of $\SW$, conforms a set of finite Hausdorff dimension in $\HI(\SD)$.


\subsection{A genericity result}

Our first theorem states that, once one restricts to the subspace $\HI^{\niso}(\SD)$, there are enough deformations to guarantee a ``\textsl{generic}'' picture:

\begin{theorem} \label{thm:generic}
Let $M$ be a $4$--manifold. Let $K$ be a closed manifold and fix a map $\SD: K \to \Engel(M)$. Denote by $\SW(k)$ the kernel of $\SD(k)$, $k \in K$.

Let $\gamma: K \to \Imm(M)$ be given satisfying $\gamma(k) \in \HI^{\niso}(\SD(k))$. Then, after a $C^\infty$--perturbation, it can be assumed that the set
\[ \{ (k, s) \in K \times \NS^1 \text{ ; } \gamma(k)'(s) \in \SW(k)_{\gamma(k)(s)} \} \]
is a submanifold of codimension $1$ in $K \times \NS^1$ in generic position with respect to the foliation with leaves $\{k\} \times \NS^1$.
\end{theorem}

We shall dedicate the rest of this subsection to its proof. It relies on local $C^\infty$--small deformations using Lemmas \ref{lem:transverseModel} and \ref{lem:tangentModel}. 

\textsl{Setup}. Let us construct an adequate cover of the space $K \times \NS^1$. Locally, for every $(k,s) \in K \times \NS^1$, we can find vector fields $W(k',s)$ spanning $\SW(k')$. Denote by $\eta(k')$ the integral curve of $W(k',s)$ with $\eta(k')(s) = \gamma(k')(s)$.

We can apply Lemma \ref{lem:tangentModel} parametrically to the curves $\eta(k')|_{[s-\varepsilon,s+\varepsilon]}$. If $\gamma(k)'(s) \in \SW(k)$, there is a product neighbourhood $\D_\varepsilon(k) \times [s-\varepsilon,s+\varepsilon]$ in which every curve $\gamma(k')$, $k' \in \D_\varepsilon(k)$, is graphical over $\eta(k')$ in said models. We say that this neighbourhood is of type I.

Otherwise, if $(k,s)$ is such that $\gamma(k)'(s)$ is transverse to $\SW(k)$, so are the nearby curves. We use Lemma \ref{lem:transverseModel} parametrically to yield a product neighbourhood of $(k,s)$ in which the curves $\gamma(k')$ look like the zero section in $J^2(\R, \R)$. We call this a neighbourhood of type II.

Then, by compactness of $K \times \NS^1$, we can find a finite cover $\{U_{i,j}\}$ comprised of neighbourhoods like the ones we just described. We assume that it is the product of a covering $\{W_i\}$ in $K$ and a covering $\{V_j = \Op([\dfrac{j}{N}, \dfrac{j+1}{N}])\}$, $j=0,..,N-1$, in $\NS^1$. We order the neighbourhoods $\{U_{i,j} = W_i \times V_j\}$ as follows: for any fixed $W_i$, we find some $j_i \in \{0,..N-1\}$ such that $W_i \times V_{j_i}$ is of type II and we order the $W_i \times V_j$ cyclically increasing from $j=j_i+1$ to $j=j_i$. The order in which we consider each $W_i$ is not important and hence we just proceed as we numbered them. See Figure \ref{fig:order}.

\begin{figure}[ht] 
\centering
\includegraphics[scale=0.7]{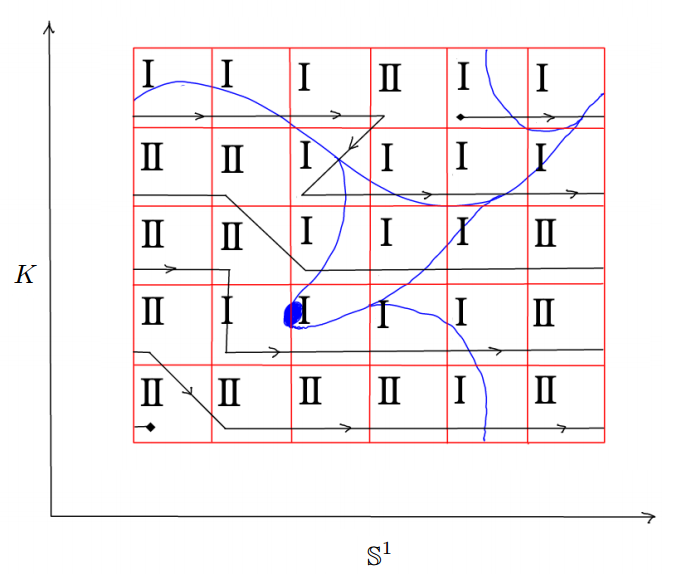}
\caption{In red we depict the manifold $K \times \NS^1$. The tangencies with $\SW$ are the points in blue. The sets $U_{i,j}$ correspond to neighbourhoods of the smaller red squares; they have been numbered as I or II depending on whether they are of the first or second type. The black line with arrows indicates the order in which we proceed for the induction.}
\label{fig:order}
\end{figure}

The idea now is to modify $\gamma$ over the neighbourhoods $U_{i,j}$ inductively using the order we just constructed. Over those of type I we will deform to achieve the desired transversality. Over those of type II we have more flexibility, so we shall use them to ensure that the deformation $\tilde\gamma$ does close up.

Take a neighbourhood $U_{i,j}$. Denote by $\tilde U_{i,j}$ the union of the neighbourhoods over which a $C^\infty$--close deformation $\tilde\gamma$ of $\gamma$ has been defined already. 

\textsl{Type I neighbourhoods.} Assume that $U_{i,j}$ is of type I. Applying Lemma \ref{lem:tangentModel}, we have a family of curves
\[ \gamma(k): V_j \to (\R^4, \ker(dy-tdx) \cap \ker(dz-t^2dx)), \qquad k \in W_i \]
that are graphical over the $x$ axis and thus given by functions:
\[ \gamma(k)(s) = (x_k(s),y_k(x_k),z_k(x_k),t_k(x_k)) \]
with $x_k(s)$ a diffeomorphism with its image, $t_k(x)$ some arbitrary function, and
\begin{equation}\label{eq:integrals}
	\left\{\begin{aligned} 
	\quad y_k(x_k(s)) = y_k(x_k(-1))+\int_{x_k(-1)}^{x_k(s)} t_k(x) dx \\
	\quad z_k(x_k(s)) = z_k(x_k(-1))+\int_{x_k(-1)}^{x_k(s)} t^2_k(x) dx 
	\end{aligned} \right.
\end{equation}
where the dependence on $k$ is smooth. Analogously, $\tilde\gamma$ is defined by functions $(\tilde x_k,\tilde y_k,\tilde z_k, \tilde t_k)$ which are only defined over $U_{i,j} \cap \tilde U_{i,j}$.

Tangencies with $\SW$ are given by $t'_k, \tilde t'_k = 0$. We extend $\tilde t_k$ from $U_{i,j} \cap \tilde U_{i,j}$ to the whole of $U_{i,j}$ arbitrarly, ensuring that it remains $C^\infty$--close to $t_k$ and that it has generic critical points (for a family of dimension $\dim(K)$). We can extend $\tilde y_k$ and $\tilde z_k$ to the whole of $U_{i,j}$ using the integral expressions \eqref{eq:integrals} with initial values those in $U_{i,j} \cap \tilde U_{i,j}$. The order that we chose for the induction means that $U_{i,j} \cap \tilde U_{i,j} \cap (\{k\} \times \NS^1)$ is connected at every such step, so in particular we are defining $\tilde y_k$ and $\tilde z_k$ as integrals with boundary conditions given only at one end of the interval $V_j$. Note that this construction is indeed relative to $\tilde U_{i,j}$.

\textsl{Type II neighbourhoods.} Assume that $U_{i,j}$ is of type II. In its local model, given by Lemma \ref{lem:transverseModel}, the perturbations $\tilde\gamma(k)$ (which are defined only over $\tilde U_{i,j}$) can be assumed to be graphical over $\gamma(k)$, which are seen as intervals contained in the zero section of $J^2(\R, \R)$. $\tilde\gamma$ is thus described by a family of functions $\tilde y_k$ and their first and second derivatives $\tilde z_k$ and $\tilde t_k$, respectively. Extend $\tilde y_k$ arbitrarily to $U_{i,j}$ while keeping it $C^\infty$ close to $y_k$. Here the boundary conditions do not yield integral constraints, we simply take $\tilde z_k$ and $\tilde t_k$ to be the corresponding derivatives of $\tilde y_k$; this is the reason why, for each $W_i$, we left $W_i \times V_{j_i}$, a neighbourhood of type II, for last. No additional tangencies with $\SW$ are introduced doing this. \hfill{$\Box$}

\begin{remark}
In type II neighbourhoods, after extending $\tilde y_k$, one could construct a bump function $\psi: U_{i,j} \to \R$ that is identically $1$ near $\partial U_{i,j}$ and identically zero in a slightly smaller ball and then take $\psi \tilde y_k$ and its derivatives as the desired extensions to the whole of $U_{i,j}$. In this manner, by taking the cover to be fine enough, one can strengthen Theorem \ref{thm:generic} saying that the deformation $\tilde\gamma$ can be taken to agree with $\gamma$ in an arbitrarily large closed set disjoint from the tangencies.
\end{remark}

\subsection{The main theorem} \label{ssec:main}

Let us state our main result:

\begin{theorem} \label{thm:main}
Let $(M, \SD)$ be an Engel manifold. The inclusion $\HI^\niso(\SD) \to \FHI(\SD)$ is a weak homotopy equivalence.
\end{theorem}

\begin{remark} \label{rem:main}
We shall prove the following slightly stronger result: Let $K$ be some compact $m$--dimensional manifold, possibly with boundary, and fix $\SD: K \to \Engel(M)$. Then, any map $\phi: K \to \FImm(M)$ with $\phi(k) \in \HI^\niso(\SD(k))$ for $k \in \partial K$ and $\phi(k) \in \FHI(\SD(k))$ for all $k$, is homotopic, relative to $\partial K$, to a map $\tilde\phi: K \to \FImm(M)$ with $\tilde\phi(k) \in \HI^\niso(\SD(k))$.
\end{remark}

Most of the work needed for the theorem is contained in the following proposition, which states that a parametric, relative (with respect to some some subset $B$ of the interval), relative in the parameter (with respect to some subset $A$ of the parameter space $K$), $C^0$--close h--principle holds for horizontal immersions of the interval.

\begin{proposition} \label{prop:main}
Let $(M = \R^3 \times (-\varepsilon, \varepsilon), \SD = \SD_\std)$. Let $A \subset \partial\D^m$ be some closed CW--complex. Let $B \subset [0,1]$ be either $\{0,1\}$, $\{0\}$ or the emptyset. Fix a map $\psi_k \in \FHI(\Op([0,1]),\SD)$, $k \in \Op(\D^m)$, conforming to the following properties:
\begin{itemize}
\item $\psi_k \in \HI^\niso(\Op([0,1]),\SD)$ for $k \in \Op(A)$.
\item $\psi_k$ is horizontal with respect to $\SD$ for $s \in \Op(B)$.
\end{itemize}

Then, there is a homotopy $\psi^\delta_k \in \FHI(\Op([0,1]),\SD)$, $\delta \in [0,1]$, starting at $\psi_k^0 = \psi_k$, such that:
\begin{itemize}
\item $\psi_k^1|_{[0,1]} \in \HI^\niso([0,1],\SD)$ for $k \in \D^m$,
\item $\psi_k^\delta = \psi_k$ for $k \in \Op(A)$ or $s \in \Op(B)$,
\item $\psi_k^\delta = \psi_k$ away from $\D^m \times [0,1]$,
\item writing $(\gamma_k^\delta, F_k^\delta)$ for the two components of $\psi_k^\delta$, $\gamma_k^\delta$ is $C^0$--close to $\gamma_k^0$.
\end{itemize}
\end{proposition}

\textbf{Remark:} We define the domain to be $\Op(\D^m) \times \Op([0,1])$ so that $\psi^\delta_k$ glues with $\psi_k$ away from $\D^m \times [0,1]$. This is important when proving a relative statement. In this direction, by $\Op(A)$ we mean an arbitrarily small neighbourhood of the radial projection of $A$ to $\partial\Op(\D^m)$ that still contains $A$. The definition of $\Op(B)$ is analogous.

Let us explain how to deduce our main theorem using Proposition \ref{prop:main}.

\begin{proof}[Proof of Theorem \ref{thm:main}]
We will prove Remark \ref{rem:main} instead. Denote by $\tilde K$ the complement of a small collar neighbourhood of $\partial K$ such that $\phi(k) \in \HI^\niso(\SD(k))$ for all $k \notin \tilde K$. After applying Theorem \ref{thm:generic}, we can assume that the curves $\phi(k)$, $k \in \Op(\partial\tilde K)$, are in general position with respect to $\SW(k)$. Find a triangulation for $\partial\tilde K \times \NS^1$ and use Thurston's jiggling to put it in general position with respect to the foliation $\SF = \coprod\{k\} \times \NS^1$. We extend this triangulation to $\tilde K \times \NS^1$ and then we apply jiggling once more. The resulting triangulation we denote by $\ST$. Write $\pi_K$ and $\pi_{\NS^1}$ for the projections of $K \times \NS^1$ to its factors.

Using $\ST$ as a tool, we find a cover of $K \times \NS^1$ by flowboxes of $\SF$ that is well suited to applying Proposition \ref{prop:main} iteratively. Refer to Figures \ref{fig:tri2d} and \ref{fig:tri3d} for a pictorial depiction and see \cite[Prop. 29]{CPPP} for a detailed proof. Regard $\ST$ as a collection of simplices $\{\tau\}$. We claim that there is a collection of open sets $\{S(\tau)\}_{\tau \in \ST}$ in $K \times \NS^1$ satisfying that:
\begin{enumerate}
\item $\sigma \subset \cup_{\tau \subset \sigma} S(\tau)$ and $\partial\sigma \subset \cup_{\tau \subsetneq \sigma} S(\tau)$,
\item if $S(\tau) \cap S(\sigma) \neq \emptyset$, then either $\sigma \subset \tau$ or $\tau \subset \sigma$,
\item there is a diffeomorphism $g(\sigma): \D^m \times [0,1] \to S(\sigma) \subset K \times \NS^1$ with $g(\sigma)^*\langle \partial_s \rangle = \langle \partial_t \rangle$, where $t$ is the last coordinate in the domain, and $s$ is the coordinate in the $\NS^1$,
\item if $\sigma$ is not top dimensional, $g(\sigma)^{-1}(\cup_{\tau \subsetneq \sigma} S(\tau)) = \Op(A) \times [0,1]$ with $A$ some closed CW--complex lying in $\partial \D^m$.
\end{enumerate}
 The proof of this claim goes roughly as follows: we find a flowbox for the zero simplices. Each such flowbox has a vertical boundary, tangent to the foliation by lines, and a horizontal boundary. If these flowboxes are suitably chosen, every higher dimensional simplex intersects them in the vertical boundary. Proceeding inductively in the dimension, given any simplex, we consider the topological disk contained in it obtained by removing the neighbourhood of its boundary that we built in previous steps. We find a flowbox neighbourhood for this smaller disk in such a way that it intersects the previous neighbourhoods only in their vertical boundaries. This proves the claim.

\begin{figure}
\begin{minipage}[hl]{0.47\textwidth}
\centering
\includegraphics[scale=0.38]{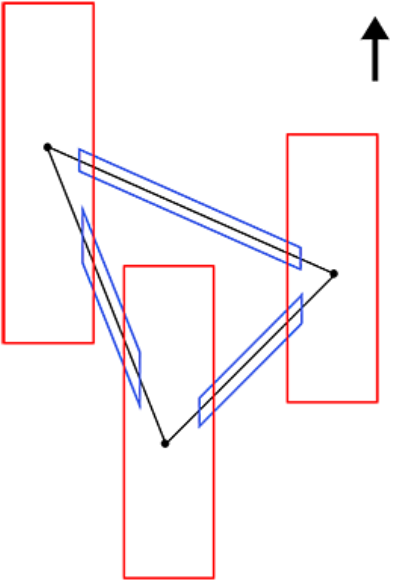}
\caption{Two dimensional example of the sets $g(\tau)(\D^m \times [0,1])$. The foliation $\coprod \{k\} \times \NS^1$ is given by the vertical lines. In red, the sets corresponding to $0$-simplices. In blue, those corresponding to $1$-simplices.}
\label{fig:tri2d}
\end{minipage}
\begin{minipage}[hr]{0.50\textwidth}
\centering
\includegraphics[scale=0.4]{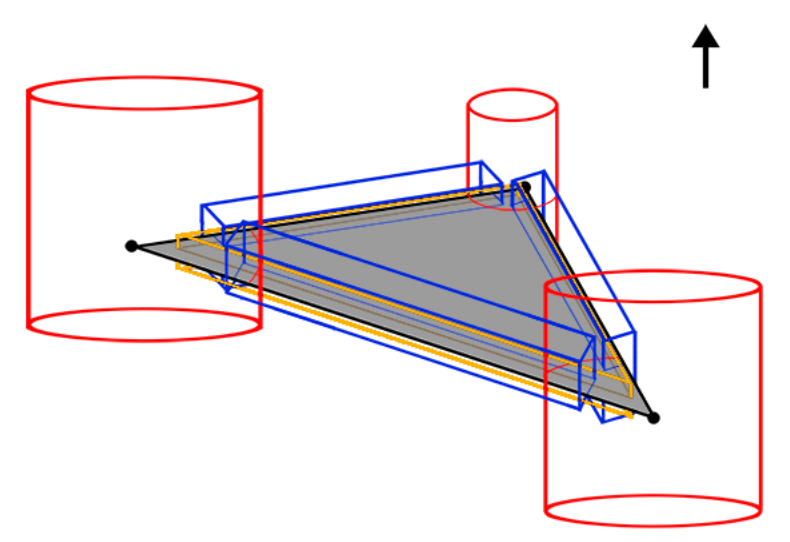}
\caption{Three dimensional example. Red: closed disks for the $0$-simplices. Blue: closed disks for the $1$-simplices. Yellow: closed disks for the $2$-simplices.}
 \label{fig:tri3d}
\end{minipage}
\end{figure}

We shall deform $\phi$ over the neighbourhoods $g(\sigma)(\D^m \times [0,1]) \subset K \times \NS^1$, for those $\sigma$ not contained in $\partial\tilde K$. We proceed inductively on the dimension of $\sigma$. Observe that, since $\ST$ was a jiggling of an arbitrary triangulation, we can assume that $\ST$ is fine enough so that $s \to \phi(k)(s)$, for $(k,s) \in g(\sigma)(\D^m \times [0,1])$, maps into a Darboux ball $B_k$ for $\SD(k)$. If the triangulation was fine enough, we can parametrically identify the balls $B_k$ with the $k$--independent Darboux ball $(M = \R^3 \times (-\varepsilon, \varepsilon), \SD = \SD_\std)$.

Now we apply Proposition \ref{prop:main} to the map $\psi_k(s) = \phi(\pi_K \circ g(k,s))(\pi_{\NS^1} \circ g(k,s))$. If $\dim(\sigma) < m+1$, we take $B=\emptyset$ and $A$ as in the enumeration above. If $\sigma$ is top dimensional, we take $A = \partial \D^m$ and $B = \{0,1\}$.
\end{proof}

Now, the strategy of proof for Proposition \ref{prop:main} is quite similar in spirit to Geiges' proof of Proposition \ref{prop:Geiges}. We first want to use the h--principle for legendrian immersions to do some work for us. Let us recall its statement:

\begin{proposition} \label{prop:legendrianhPrinciple}
A $C^0$--dense, parametric, relative, and relative to the parameter $h$--principle holds for legendrian immersions in contact manifolds. 
\end{proposition}

What this means precisely could be spelt out explicitely much in the fashion of Proposition \ref{prop:main}. This auxiliary result will be enough for the lower dimensional cells, but for the top ones we will need some adjustments. We break down the proof of Proposition \ref{prop:main} in several steps.

\textsl{Step I.} The image of $M = \R^3 \times (-\varepsilon, \varepsilon)$ under the Geiges projection $\pi_\Geiges$ is $V= \R^2 \times (-\varepsilon, \varepsilon)$ with coordinates $(x,z,t)$. Horizontal immersions descend to legendrian immersions for the standard contact structure $\xi = \ker(dz-tdx)$. In particular, tangencies with $\SW$ upstairs are in correspondence with tangencies downstairs with $\langle \partial_t \rangle$. From this, it follows that, whenever $\psi_k$ is horizontal and generic ($k \in \Op(A)$ or $s \in \Op(B)$), $\pi_\Geiges \circ \psi_k$ is in general position with respect to the legendrian foliation given by $\langle \partial_t \rangle$, and thus the singularities of its front are generic. Do note that, since we work with higher dimensional families, singularities more complicated than cusps do appear.

Let us denote $\Leg(V, \xi)$ for the legendrian immersions of the interval $\Op([0,1])$ into $(V, \xi)$ and $\FLeg(V, \xi)$ for its formal counterpart. Much like in the case of horizontal immersions, a formal legendrian immersion is a pair comprised of a map into $V$ and a monomorphism into $\xi$ (in this case, both with domain the interval). 

Since $d\pi_\Geiges$ maps $\SD$ isomorphically onto $\xi$, the Geiges projection yields a family
\begin{align*}
\Psi^0_k = \pi_\Geiges \circ \psi_k \in &\quad \FLeg(V, \xi), & k \in \Op(\D^m), \\
\Psi^0_k \in &\quad \Leg(V, \xi), & k \in \Op(A),
\end{align*}
which is already legendrian for $s \in \Op(B)$. By Proposition \ref{prop:legendrianhPrinciple}, $\Psi_k^0$ is homotopic, relative to $A$ and $B$, to a family $\Psi_k^{1/2} \in \Leg(V, \xi)$ for all $k$. We can further assume that the front of $\Psi_k^{1/2}$ has generic singularities as well. Denote by $\Psi_k^\delta = (\eta_k^\delta, G_k^\delta)$, $\delta \in [0,1/2]$, the homotopy as formal legendrians.

\textsl{Step II.} Let us construct a lift $\psi_k^\delta = (\gamma_k^\delta, F_k^\delta)$ of $\Psi_k^\delta$. Since $\SD$ projects to $\xi$ under the Geiges projection, we define $F_k^\delta$ to be the unique lift of $G_k^\delta$. For $\gamma_k^\delta$, let us focus first on the case where $B$ is empty or $\{0\}$. Define its $y$--coordinate $y^\delta_k(s)$, for $(k,s) \in (\D^m \times [0,1]) \cup (\Op(A) \times \Op([0,1]))$, to be given by:
\[ y^\delta_k(s) = y^0_k(0) + \int_{\eta_k^\delta|_{[0,s]}} zdx. \]
In the complement, we extend $\gamma_k^\delta$ by interpolating back to $\gamma_k^0$. This construction guarantees $\gamma_k^\delta = \gamma_k^0$ for $k \in \Op(A)$.

\textsl{Step III.} If $B = \{0,1\}$, defining $y_k^\delta$ by integration means that the $y$--coordinate of $\gamma_k^\delta$ will not necessarily agree with that of $\gamma_k^0$ at $s=1$, as it should. The idea is to deform $\eta_k^{1/2}$ to yield a new Geiges projection $\eta_k^1$ having this integral adjusted. Note that we cannot do wild deformations: for a legendrian not to escape the local model $V = \R^2 \times (-\varepsilon, \varepsilon)$, its front must have a slope bounded in terms of $\varepsilon$. Instead, we introduce type I Reidemeister moves to add or substract area.

Recall that the front of $\eta_k^{1/2}$ has generic singularities. In particular, given any point $k \in \D^m$, there is $s_k$ such that the curve $s \to \eta_k^{1/2}(s)$ is not tangent to $\langle \partial_t \rangle$ at time $s_k$. It follows that we can find a small disk $\SU_k \subset \D^m$ containing $k$ and an interval $I_k \subset [0,1]$ containing $s_k$ such that the curves $s \to \psi_{k'}^{1/2}(s)$, $(k',s) \in \SU_k \times I_k$, are transverse to $\langle \partial_t \rangle$ and therefore their front projection is an interval without cusps. By compactness, a finite number of open subsets $\SU_k$ disjoint from $A$ cover $\D^m \setminus \Op(A)$.

Given any even integer $N$, find an ordered sequence of times $s_k^1,..,s_k^N \in I_k$ and a width $\epsilon > 0$ such that the segments $[s_k^j-\epsilon, s_k^j+\epsilon] \subset I_k$ do not overlap. We construct $\eta_k^1$ as follows. Replace the front of the curves $\eta_k^{1/2}|_{[s_k^j-\epsilon, s_k^j+\epsilon]}$, for $k \in \SU_k$, by adding a ``Reidemeister I'' loop such that the sign of the area it encloses is given by the parity of $j$. Modify the fronts of $\eta_k^{1/2}$, for $k \in \Op(\SU_k) \setminus \SU_k$, so that they transition, through Reidemeister I moves, from agreeing with those of $\eta_k^{1/2}$ in $\partial\Op(\SU_k)$ to agreeing with those of $\eta_k^1$ in $\SU_k$. Denote by $\eta_k^\delta$, $\delta \in [1/2,1]$, the corresponding legendrian homotopy.

A remark is in order. The slopes of the fronts of $\eta_k^\delta$, $\delta \in [1/2,1]$, can be assumed to remain arbitrarily close to those of $\eta_k^{1/2}$; in particular, the deformation does not escape the Darboux ball $M$. In particular, we can find a bound, independent of $N$ but depending on how much we want to $C^0$--approximate $\eta_k^{1/2}$, for how large the areas enclosed by the Reidemeister I loops can be. This implies that we can adjust $N$ and the size of the loops to modify the area to be exactly the amount we require.

Since $\eta_k^\delta$ is legendrian for $\delta \in [1/2,1]$, its tangent map extends $G_k^\delta$ to the whole of $\delta \in [0,1]$. $G_k^\delta$ lifts to $F_k^\delta$ as above. We define $\psi_k^1$ (or rather its $y$--coordinate) by integrating $zdx$ over $\eta_k^1$. Since the $\SU_k$ cover $\D^m$, we have that for all $k \in \D^m$ this integral can be adjusted to ensure $\psi_k^1(1) = \psi_k^0(1)$. We define the $y$--coordinate of $\psi_k^\delta$ by lifting it arbitrarily but relative to $s=0,1$. \hfill{$\Box$}

An immediate consequence is an extension of the Adachi--Geiges result to any Engel manifold:

\begin{corollary}
Let $(M, \SD)$ be an Engel manifold and let $\gamma_1,\gamma_2 \in \HI^\niso(\SD)$ be two horizontal loops. Then, they are isotopic as horizontal loops if and only if they are homotopic as maps and they have the same rotation number.
\end{corollary}
\begin{proof}
Apply Theorem \ref{thm:main} to obtain a connecting family of immersions. One can then proceed in a cover by Darboux charts, much like in Theorem \ref{thm:generic}, in which intersection points, under the Geiges projection, appear as self--tangencies satisfying an area condition. Generically, curves with self--tangencies can be assumed to be isolated in a $1$--parametric family. By adding or substracting area around said points, they can be assumed not to lift to intersections.
\end{proof}

\section{Curves tangent to $\SW$ and their deformations} \label{sec:shortCurves}

In Subsection \ref{ssec:BH} we showed that the $h$--principle for horizontal immersions does fail, in general, in the presence of closed orbits of the kernel of the Engel structure. This motivated us to restrict, in Section \ref{sec:main}, to the subclass of curves that are not everywhere tangent to the kernel. Going in the opposite direction, in this section we explore the phenomenon of rigidity in more detail.


Let us explain our setup. Take the standard $(\R^3, \xi = \ker(dy-zdx))$ and let $\phi: (\R^3, \xi) \to (\R^3, \xi)$ be a contactomorphism that fixes the origin. Think about the mapping torus $M_\phi$ as the quotient $\R^3 \times [0,1]/ \phi$ with coordinates $(x,y,z,t)$. $M_\phi$ can be endowed with a natural even contact structure: the pull--back of $\xi$, whose kernel is spanned by $\partial_t$. An Engel structure $\SD = \langle \partial_t, L\rangle$ can be defined on $M_\phi$, where $L \subset \xi$ is some $t$--dependent Legendrian vector field rotating positively in the $t$--direction and satisfying $\langle \phi^*(L(0))\rangle = \langle L(1) \rangle$.

Fix a framing $\langle X = \partial_x + z\partial_y, Z = \partial_z \rangle$ of $\xi$. Fix $L(0) = X$. We write $L(1)$ as $\cos(F(x,y,z))X + \sin(F(x,y,z))Z$, where $F: \R^3 \to \R$ is the smallest possible such function that is still positive. $F$ can be extended to the whole mapping torus to define a possible $L$: 
\[ f: \R^3 \times [0,1] \to \R \]
\[ f|_{\R^3 \times \{1\}} = F, \qquad f|_{\R^3 \times \{0\}} = 0, \qquad \partial_t f > 0, \]
\[ L = \cos(f(x,y,z,t))X + \sin(f(x,y,z,t))Z. \]
Therefore, the structure equations for the Engel structure $\SD$ read as:
\[ \alpha = dy-zdx, \qquad \beta = \cos(f(x,y,z,t))dz - \sin(f(x,y,t,z))dx. \]

Consider the $\SW$--integral curve $\gamma(\theta)=(0,0,0,\theta)$. Any $C^1$--small deformation of $\gamma$ is of the form $\eta(\theta) =  (x(\theta), y(\theta), z(\theta), \theta)$, and satisfies the equations:
\begin{eqnarray*}
\tan(f(x,y,z,t)) &= & \frac{z'}{x'}, \\
y(t)-y(0) &= & \int_0^{t} zx' ds, \\
\phi(\eta(1)) & =& \eta(0).
\end{eqnarray*}
We say that the plane curve $\pi \circ \eta(\theta) =  (x(\theta), z(\theta))$ is the front of $\eta$. These formulas in particular describe how to recover $\eta$ from its front. Using this language, we can provide a more geometrical proof of Proposition \ref{prop:BH2}.

\begin{proof}[Alternate proof of Proposition \ref{prop:BH2}]
Given a $C^1$--perturbation $\eta$ of a $\SW$--tangent curve $\gamma$, we want to show that $\eta$ is tangent to $\SW$ as well. Suppose otherwise; by Theorem \ref{thm:generic}, we can assume that $\eta$ is in general position with respect to $\SW$. We can find a neighbourhood of $\gamma$ that is a mapping torus $M_\phi$ with $\phi$ the identity and $L(1) = -L(0)$; we are in the setup above, with $f(x,y,z,t) = \pi t$. The front $\pi \circ \eta$ is a closed plane curve with cusps. It must possess at least one cusp and, choosing our neighbourhood suitably, $\pi \circ \eta$ has, at $\theta=0$, a cusp pointing to the left located at the origin. The first equation above states that the slope of $\eta$ rotates clockwise $\pi$ degrees, and thus the curve is piecewise convex. The second one says that the signed area bounded by $\eta$ must be zero.

Observe that the number of cusps must be odd since the oriented slope approaching $t=\pi$ must be horizontal and pointing to the left and at every cusp the orientation changes sign. Denote the values of the parameter for which the curve has a cusp by $\{t_0=0=\pi, t_1, \ldots, t_{2n} \}$. Since the slope is only horizontal at the endpoints, the cusps are alternating; that is, at $t_{2i-1}$ the curve leaves the horizontal line $\{ z= z(t_{2i-1}) \}$ going downwards and at $t_{2i}$ it leaves it going upwards. In other words, the function $z(t)$ is strictly increasing in the intervals $(t_{2i}, t_{2i+1})$ and strictly decreasing otherwise. We deform the curve by {\em enlarging the cusps}: Except for the one at the origin, we add a straight segment to the end of each of the cusps and then we make it convex by a slight deformation (like glueing a thickened needle to its end). This procedure allows us, without changing the total area, to push upwards the odd cusps and push downwards the even ones. Therefore, for $i>0$:
\begin{eqnarray*}
z(t_{2i-1})&=& z(t_1) > 0, \\
z(t_{2i}) &=& z(t_2) < 0.
\end{eqnarray*}

Consider the segments $\pi \circ \eta|_{(t_{2i-1}, t_{2i})}$ and $\pi \circ \eta|_{(t_{2i}, t_{2i+1})}$, and reverse the parametrisation of the former. Then, both of them are segments starting from $\pi \circ \eta(t_{2i})$ and finishing in the same $z$--coordinate, but the latter has greater slope. This reasoning readily implies that:
\[ x(t_1) > x(t_3) > \cdots > x(t_{2n-1}), \]
\[ x(t_2) < x(t_4) < \cdots < x(t_{2n}). \]
In particular, the segments $\pi \circ \eta|_{(t_{2i-1}, t_{2i})}$ and $\pi \circ \eta|_{(t_{2i+1}, t_{2i+2})}$ intersect at a point $s_i$. This means that inbetween $t_{2i-1}$ and $t_{2i+2}$ a \textsl{Reidemeister I move} configuration appears, bounding some positive area. Refer to Figure \ref{fig:BH2}.

\begin{figure}[ht] 
\centering
\includegraphics[scale=0.65]{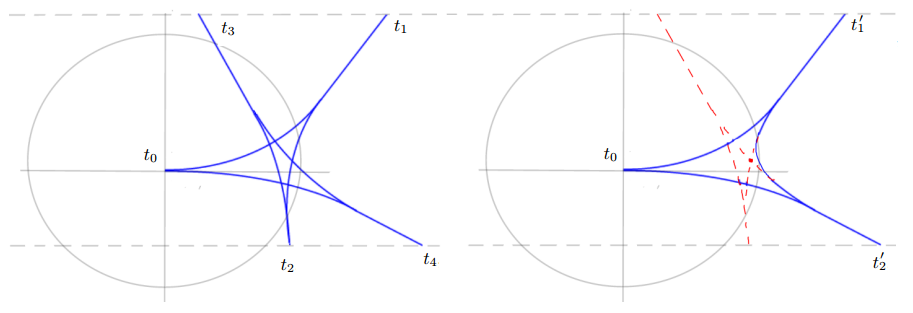}
\caption{On the left hand side, a possible projection for a deformation $\eta$ with five cusps. On the right hand side, we outline in red the area of the Reidemeister I loop, that has been removed, yielding a curve with only three cusps. Note that the cusps have been made longer so that they would reach the horizontal gray lines.}
\label{fig:BH2}
\end{figure}

Now we conclude by induction on $2n+1$, the number of cusps. Our induction hypothesis is that a front conforming to the properties above must bound positive area. This is straightforward for $2n+1=3$. For the induction step, the reasoning on the previous paragraph shows that, for $2n+1>3$, a Reidemeister I move appears. By removing it (along with the points $t_{2i}$ and $t_{2i+1}$) and smoothing the curve at $s_i$, the points $t_{2i-1}$ to $t_{2i+2}$ are now connected by a segment with no cusps. Since the area under this operation decreases and now the number of cusps is $2n-1$, the induction hypothesis concludes the proof.
\end{proof}

On the other hand, we now describe two examples of ``short'' $\SW$--orbits that admit deformations not everywhere tangent to $\SW$. These models can be inserted into Cartan prolongation by deforming $\SW$ using some contact vector field. 

\begin{example}[Curves making one projective turn.]
Take the mapping torus $M_\phi$ with $\phi(x,y,z)= (x, y/2, z/2)$. Let $\eta$ be the desired deformation, which we assume is in general position with respect to $\SW$. Its front $\pi \circ \eta(\theta) = (x(\theta),z(\theta))$ satisfies $(x(0), z(0)) = (x(1), 2z(1)) = (0,0)$ and encloses an area of $y(1)/2$. On the left hand side of Figure \ref{fig:deformation}, such a curve is presented; it is clear that the area it bounds can be adjusted to be exactly $y(1)/2$.\hfill{$\Box$}
\end{example}

\begin{figure}[ht] 
\centering
\includegraphics[scale=0.55]{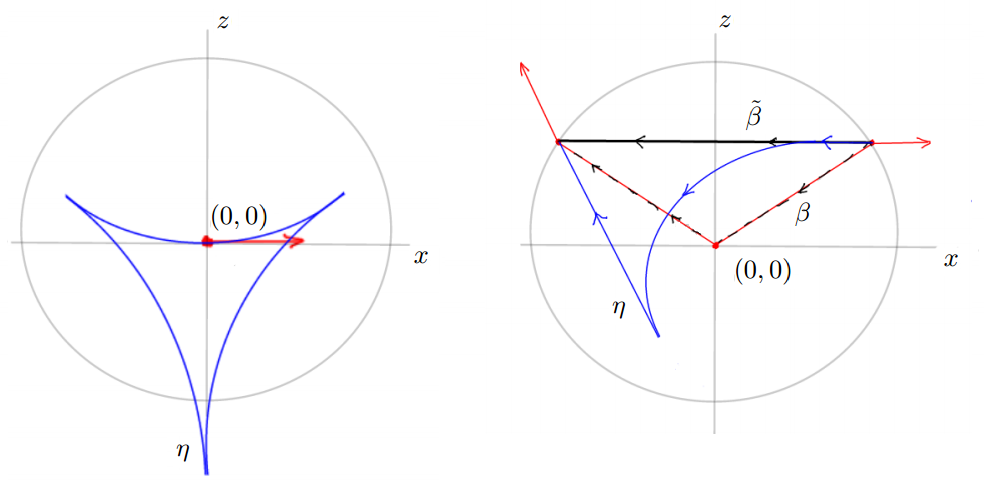}
\caption{On the left hand side, a possible deformation for a curve making one projective turn. On the right, a deformation for a curve with short developing map. The curves are depicted in blue. The tangent vectors at $t=0,1$ are shown in red.}
\label{fig:deformation}
\end{figure}

\begin{example}[Curves having an arbitrarily short developing map.]
Fix some angle $\alpha \in (0,\pi)$. The following contactomorphism is the lift of the turn of angle $-\alpha$ in the plane $(x,z)$:
\[ \psi(x,y,z)= (\cos(\alpha)x+\sin(\alpha)z, y - \sin^2(\alpha)zx + \dfrac{1}{2}\cos(\alpha)\sin(\alpha)(z^2-x^2), \cos(\alpha)z-\sin(\alpha)x). \]
We consider the mapping torus of $\psi$. 

Take a deformation $\eta$ ending at $(x(1), y(1), z(1))$. The projection $\pi\circ\eta$ must bound a signed area of 
\[ y(1) - y(0) = \sin^2(\alpha)z(1)x(1) - \dfrac{1}{2}\cos(\alpha)\sin(\alpha)[z(1)^2-x(1)^2]. \]
The right hand side is precisely the integral of $zdx$ over the curve $\beta$ given by going from $(x(0),z(0))$ to the origin and then to $(x(1),z(1))$ following straight lines, as a computation shows.

Consider $(x(1),z(1))$ lying in the first quadrant and making an angle of $\alpha/2$ with the vertical axis. Let $\tilde\beta$ be the straight horizontal segment connecting $(x(0),z(0))$ and  $(x(1),z(1))$. In particular, it lies above $\beta$ and thus $\int_\beta zdx > \int_{\tilde\beta} zdx$. Now it is straightforward to create a curve $\eta$ such that $\int_\eta zdx = \int_\beta zdx = y(1)-y(0)$ by adding some (positive) area to $\tilde\beta$. Refer to the right hand side of Figure \ref{fig:deformation}.\hfill{$\Box$}
\end{example}

Slightly generalising the first example, it is not hard to show that:
\begin{proposition} \label{prop:deformations}
Let $\phi$ be a contactomorphism of $(\R^3, \xi = \ker(dy-zdx))$ fixing the origin and with conformal factor different from $1$. Let $M_\phi$ be the corresponding mapping torus with coordinates $(x,y,z,t)$ and endowed with the Engel structure with smallest turning. Then, the $\SW$--curve $\gamma(\theta) = (0,0,0,\theta)$ admits deformations somewhere not tangent to $\SW$.
\end{proposition}
\begin{proof}
Take $d_0\phi$, the linearisation at the origin. $d_0\phi|_\xi$ is a linear map in $\R^2$ that can be lifted to a contactomorphism $\tilde\phi$. By zooming in with the contactomorphism $(x,y,z) \to (\lambda x, \lambda^2 y, \lambda z)$, $\phi$ becomes $C^\infty$ close to $\tilde\phi$, and therefore it is enough to prove the statement for $\phi$ linear.

If the conformal factor at the origin is different from $1$, there is a dilation in the $y$--coordinate. Then we construct a deformation starting and finishing at the origin and bounding an area $y(1)-y(0)>0$, which is possible if we select $y(0)$ small enough and with the adequate sign.
\end{proof}

Let us elaborate on an interesting consequence of Proposition \ref{prop:deformations}. Observe that, at a linear level, contactomorphisms fixing the origin and having conformal factor different from $1$ are generic. This readily implies a Kupka--Smale type of theorem for kernels of even--contact structures. Let us spell it out:

\begin{theorem} \label{thm:main2}
A $C^\infty$--generic even--contact structure has isolated $\SW$--orbits having Poincar\'e map not a strict contactomorphism.

The same holds for a generic Engel structure. In particular, the inclusion $\pi_0(\HI(\SD)) \to \pi_0(\FHI(\SD))$ is a bijection if $\SD$ is $C^\infty$--generic.
\end{theorem}

A key ingredient in the proof will be the following standard fact, whose proof we recall:

\begin{lemma} \label{lem:genContacto}
Let $(\D^3, \xi = \ker(\alpha))$ be the standard contact Darboux ball. Consider the space of maps from $\D^3$ to itself that are contact (but not necessarily bijective). Then, the subset of those having non--degenerate fixed points is open and dense.
\end{lemma}
\begin{proof}
Consider the manifold $V = \R^3 \times \R^3 \times \R$, and let $\pi_1$ and $\pi_2$ be the projections onto its first and second factors, respectively. $V$ can be endowed with the contact structure $\ker(\lambda = \pi_1^* \alpha - e^t\pi_2^*\alpha)$; any contact map $\phi: \D^3 \to \D^3$ lifts to a legendrian $\Gamma_\phi(x) = (x,\phi(x), \log[\alpha/(\phi^*\alpha)])$, where the last term accounts for the conformal factor of $\phi$. We need for $\Gamma_\phi$ to be transversal to $\Delta \times \R$, with $\Delta \subset \R^3 \times \R^3$ the diagonal.

This follows from Thom's transversality (see, for instance, \cite[p. 17, 2.3.2]{EM}). Indeed, let $p \in \Gamma_\phi \cap (\Delta \times \R)$. Then, there is a neighbourhod $U \ni p$ contactomorphic to $J^1(\R^3, \R)$ with $\Gamma_\phi \cap U$ being taken to the zero section. Then, Thom's transversality states that a generic $C^\infty$--small deformation of $\Gamma_\phi \cap U$ (which is given as the graph of a function) is transversal to the submanifold $(\Delta \times \R) \cap U$. Proceeding chart by chart, capping the deformations off, and using progressively smaller deformations allows us to conclude. Since the deformations are $C^\infty$-small, they are graphical over the first factor of $V$, and hence give rise to a contact map.

Note that reasoning in this fashion yields the analogous result for contactomorphisms in compact contact manifolds of any dimension as well.
\end{proof}

\begin{proof}[Proof of Theorem \ref{thm:main2}.]
Fix a metric on $M$. For simplicity, focus on even--contact structures having orientable and oriented kernel. Any such $\SE$ has an associated unitary vector field $W$ spanning $\SW$ positively. Consider the subset of even--contact structures such that the $W$--orbits of length at most $T > 0$ are non degenerate. We claim that it is open and dense. We claim that it is still open and dense if we further require for the Poincar\'e return maps of the orbits to be non--strict contactomorphisms. Assuming these statements, the subset of even--contact structures such that this is true for orbits of all periods is a countable intersection of open and dense sets.

Our claims readily follow from arguments of Peixoto \cite{Pe}[p. 219-220], which we briefly sketch. Take $(M, \SE)$. Given any $W$--orbit $\gamma$ of period $\tau < T$, Lemma \ref{lem:genContacto} produces a $C^\infty$--small deformation of $W$ such that the Poincar\'e return has only isolated fixed points. However, this might produce new orbits of period $2\tau-\varepsilon \geq N\tau < T$ for some integer $N$. Therefore, one starts deforming orbits that are close to the minimal period and introduces progressively smaller deformations as the period goes up to $T$. If we additionally want the orbits not to have return map a strict contactomorphism, we take the isolated orbits we have produced and we replace their Poincar\'e return maps by their linearised version, which we then make generic. 

This concludes the proof for the statement regarding even--contact structures. We then note that any $C^\infty$--perturbation of $\SE = [\SD, \SD]$ can be realised by a $C^\infty$--perturbation of $\SD$, so we conclude that the same holds for a $C^\infty$--generic Engel structure. Having all $\SW$--orbits isolated, Proposition \ref{prop:deformations} allows us to deduce that, at least at the $\pi_0$ level, there is a complete $h$--principle for the inclusion $\HI(\SD) \to \FHI(\SD)$ if $\SD$ is generic.
\end{proof}

\begin{remark}
Theorem \ref{thm:main2} should still be true for higher $\pi_k$. This would require carefully analysing families of curves and ensuring that the model from Figure \ref{fig:deformation} can be introduced parametrically.
\end{remark}

\section{Transverse maps and immersions} \label{sec:transverse}

Having proven our results on horizontal immersions, we can study the other condition that is geometrically meaningful for a map to satisfy in the presence of a distribution: that of being transverse. We shall review Gromov's strategy for proving flexibility. This was already worked out in detail by Y. Eliashberg and N. Mishachev in \cite[p. 136]{EM} for the contact case, and indeed the proof goes through without any major differences.

\begin{theorem} \label{thm:transverse}
Let $(M, \SD)$ be an Engel manifold and $V$ be any manifold. Maps $f: V \to M$ with $df: TV \to TM \to TM/\SD$ surjective satisfy a $C^0$--close, parametric, relative, and relative to the parameter $h$--principle.
\end{theorem}

\textbf{Remark:} If $V$ is $2$--dimensional, the statement amounts to asking for $V$ to be an immersion transverse to $\SD$. The analogous statement for $V$ having subcritical dimension $1$ (or $0$) is already proven in \cite{EM}[Prop. 8.3.2]. 

\textbf{Remark:} assume that the Engel flag $\SW \subset \SD \subset \SE \subset TM$ is orientable. Then, if $V$ is an immersed closed transverse $2$--dimensional manifold, it must be a torus with trivial normal bundle. If we drop the orientability assumption, $V$ can be a Klein bottle as well. 

\textbf{Remark:} the formal data is a mapping $f$ and a formal derivative $F: TV \to TM$ that is surjective onto the quotient $TM/\SD$.

\subsection{The $h$--principle for Diff--invariant, microflexible and locally integrable relations}

Let us explain the main ingredients needed to prove Theorem \ref{thm:transverse}. The interested reader might want to refer to \cite{EM}[Chap. 13].

Fix two manifolds $W^n$ and $M$ and let $\pi: J^r(W, M) \to W$ be the space of $r$--jets from $W$ to $M$. $r$ can take the value $\infty$ or the value $g$, by which we mean germs of maps. A subset $\SR \subset J^r(W, M)$ is called a differential relation.

\begin{definition}
A differential relation $\SR$ is \emph{locally integrable} if, for any $m$, and for any two maps
\[ h: [0,1]^m \to J^r(W, M)  \] 
\[ g_p: \Op(\pi \circ h(p)) \to M \text{, } p \in \Op(\partial[0,1]^m) \]
satisfying $(J^r \circ g_p)(\pi \circ h(p)) = h(p)$, there is 
\[ f_p: \Op(\pi \circ h(p)) \to M \text{, } p \in [0,1]^m \] 
satisfying $(J^r \circ f_p)(\pi \circ h(p)) = h(p)$ for all $p$, and $f_p = g_p$ for all $p \in \Op(\partial[0,1]^m)$.
\end{definition}

That is, $\SR$ is locally integrable if any pointwise differential condition given by $\SR$ can be locally extended to a solution. We introduce the parameter space $[0,1]^m$ to state that this local solvability holds parametrically and relatively as you vary the pointwise condition.

Let us denote $\theta_l = (A = [-1,1]^n, B = \partial([-1,1]^n) \cup ([-1,1]^l \times \{0\}))$.

\begin{definition}
A relation $\SR$ is \emph{microflexible} if, for any small ball $U \subset W$, any $m$, and any maps 
\[ h_p: \theta_l \to U \text{, } p \in [0,1]^m \text{, embeddings}, \]
\[ F_p: \Op(h_p(A)) \to \SR \text { holonomic}, \]
\[ \tilde F^t_p: \Op(h_p(B)) \to \SR  \text{, } t \in [0,1] \text{, holonomic  and satisfying } \tilde F^t_p = F_p \text{ for } p \in \Op(\partial[0,1]^m) \text{ or } t=0, \]
there is, for small $t$, a holonomic family $F^t_p: \Op(h_p(A)) \to \SR$ extending $\tilde F^t_p$, and satisfying $F^t_p = F_p$ if $p \in \Op(\partial[0,1]^m)$ or $t=0$. If the extension exists for all $t$, we say that $\SR$ is \emph{flexible}. 
\end{definition}

That is, being microflexible amounts to proving that local deformations of a solution of the differential solution can be extended to global solutions, as least for small times. Relations that are open are immediately microflexible and locally integrable.

The following proposition \cite[13.5.3]{EM} holds:

\begin{proposition}[Gromov] \label{prop:microflex}
Let $\SR \subset J^r(V \times \R, M)$ be a locally integrable and microflexible relation that is invariant with respect to diffeomorphisms that leafwise preserve the foliation $\coprod \{v\} \times \R$. Then, a $C^0$--close, parametric, relative, and relative to the parameter $h$--principle holds in $\Op(V \times \{0\})$.
\end{proposition}

Saying that the $h$--principle holds means that the space of holonomic sections (sections such that the formal derivatives are the actual derivatives of the zeroeth order map) is weak homotopy equivalent --by the inclusion-- to the space of all sections into $\SR$. Note that, by $C^0$--close it is meant that the zeroeth order components are $C^0$--close, not its derivatives.

\subsection{Proof of Theorem \ref{thm:transverse}}

Let $(M, \SD)$ be an Engel manifold. We claim that the relation $\SR_1$ in $\pi: J^1(\R, M) \to \R$ of being tangent to $\SD$ but transverse to $\SW$ is locally integrable. Suppose we are given maps
\[ h: [0,1]^m \to (\SD \setminus \SW) \subset TM  \] 
\[ g_p: \Op(0) \subset \R \to M \text{, } p \in \Op(\partial[0,1]^m) \]
where the $g_p$ are horizontal curves transverse to $\SW$ satisfying $dg_p(0) = h(p)$. For all $p \in [0,1]^m$ and depending smoothly on $p$, we can extend the vector $h(p)$ to a vector field $H_p$ in $\Op(\pi \circ h(p))$. We can assume that the maps $g_p$ are embeddings by shrinking the domain. Therefore, for those $p \in \Op(\partial[0,1]^m)$, $H_p$ can be assumed to be an extension of the tangent vector $g_p'$. Following the flow of $H_p$ for short times gives the desired local extension of $g_p$.

We claim that $\SR_1$ is microflexible as well. Observe that we only have to consider the case $\theta_0$, which can be phrased as follows. Let $F^0_p: [0,1] \to \SR_1$, $p \in [0,1]^m$, be a family of holonomic maps. Let $F^t_p: \Op(\{0,1\}) \to \SR_1$, $t \in [0,1]$, be a family of deformations defined around the endpoints of the interval. Let $\psi: [0,1] \to \R$ be a bump function which is identically $1$ around $\{0,1\}$ and zero in an arbitrarily large interval in the interior of $[0,1]$. According to Lemma \ref{lem:transverseModel}, the curves $F^0_p$ possess a local model in which they correspond to the zero section in $J^2(\R, \R)$; this implies that, for small $t$, $F^t_p$ is graphical over $F_p^0$ and therefore given by a function $y_p^0$. The extension is given by $\psi y_p^0$ and its derivatives.

Let $V$ be some manifold. Let $\SR_2 \subset J^1(V, M)$ be the open relation of having the formal derivative be surjective onto $TM/\SD$. The relation $\SR_3 \subset J^1(V \times \R, M)$ consists of those maps with formal derivative surjective onto $TM/\SD$ that, along the fibres $\{v\} \times \R$, are tangent to $\SD$ but transverse to $\SW$. Local integrability for $\SR_3$ follows by mimicking the argument for $\SR_1$. 

We claim that $\SR_3$ is also microflexible. Take $\theta_j = (A,B)$. Suppose we are given a holonomic family $F_p^0$ on $A$ and a corresponding deformation $F_p^t$ over $\Op(B)$. Find neighbourhoods $\Op_1(B) \subset \Op_2(B) \subset \Op(B)$ and build a bump function $\psi$ that is $1$ in $\Op_1(B)$ and $0$ outside of $\Op_2(B)$. Since $F_p^t$ is fibrewise graphical over $F_p^0$ for small $t$, we use $\psi$ to interpolate back to $F_p^0$, as above; this can be achieved even if $B$ is embedded wildly with respect to the foliation $\coprod \{v\} \times \R$. For small times the resulting deformation is $C^\infty$--close to $F_p^0$, so in particular it is still surjective onto $TM/\SD$ in the transverse direction.

By construction, $\SR_3$ is invariant under diffeomorphisms preserving the foliation $\coprod \{v\} \times \R$ leafwise. Then, Proposition \ref{prop:microflex} allows us to conclude that in $\Op(V \times \{0\})$ a complete $h$--principle holds, so in particular a complete $h$--principle holds in $V$ for the relation $\SR_2$. \hfill{$\Box$}

\textbf{Remark:} Observe that we did not need the $h$--principle for tangent immersions in Theorem \ref{thm:main}, instead we just checked the much more simple properties of being microflexible and locally integrable for the relation $\SR_1$.

\subsection{Immersed 3--dimensional submanifolds}

The reader might have noticed that the most interesting case for a transverse submanifold was left out: codimension 1. Inspecting the proof presented in the previous subsection, it is clear that it cannot possibly go through, since immersions $V^3 \times \R \to M$ cannot avoid the $\SW$--direction, which was a key ingredient in the 2--dimensional case to obtain microflexibility. Still, the following result holds:

\begin{proposition} \label{prop:3immersions}
Let $(M, \SD)$ be an Engel manifold. Let $V$ be an arbitrary $3$--manifold. Then, immersions $V \to M$ that are transverse to $\SD$ satisfy a $C^0$--close, relative, relative to the parameter $h$--principle.
\end{proposition}
\begin{proof}
Let us make a few remarks that will shed some light on how to proceed. Consider the space of horizontal curves:
\[ \HI^\gen(\SD) = \{ \gamma \in \HI^\niso(\SD) \text{ such that $\gamma$ is not tangent to $\SW$ identically in a neighbourhood of a point}\}. \]
Theorem \ref{thm:generic} homotopes any compact family in $\HI^\niso(\SD)$ to lie in $\HI^\gen(\SD)$ (and any compact family in $\HI^\gen(\SD)$ to one having nice singularities). Its relative nature readily implies that the inclusion $\HI^\gen(\SD) \to \HI^\niso(\SD)$ is a weak homotopy equivalence. Observe that $\HI^\gen(\SD)$ is a nicer space than $\HI^\niso(\SD)$ in the sense that the curves it contains are defined by a local condition. Define $\HI^\gen(\R, \SD)$ analogously for maps defined over $\R$. 

Let $V$ be a $3$--manifold. Define the following differential relation $\SR \subset J^g(V \times \R, M)$: germs that are transverse to $\SD$ along $V \times \{s\}$ and lie in $\HI^\gen(\R, \SD)$ along $\{v\} \times \R$. There exists an obvious projection $J^g(V \times \R, M) \to J^1(V \times \R, M)$ and the image of $\SR$ is the relation $\SR^1$: maps with formal differential transverse to $\SD$ along $V \times \{s\}$ and tangent to $\SD$ along $\{v\} \times \R$. $\SR \to \SR^1$ is a Serre fibration with contractible fibre.


The proof of Proposition \ref{prop:3immersions} amounts to showing that $\SR$ is microflexible and locally integrable and then applying Proposition \ref{prop:microflex}. The full h--principle for $\HI^\gen(\SD)$ and the openess of the transverse immersion condition in codimension $1$ imply microflexibility and the local integrability is tautological. The claim follows. 
\end{proof}

\end{document}